\documentclass[a4paper,11pt,reqno]{amsart}

\usepackage{amsmath}
\usepackage{amsfonts}
\usepackage{amssymb}
\usepackage{amsthm}
\usepackage[shortlabels]{enumitem}
\usepackage{graphicx}
\usepackage{color}
\usepackage{dsfont}
\usepackage[latin1]{inputenc}
\usepackage{MnSymbol}
\usepackage{enumitem}
\usepackage{extpfeil}
\usepackage{mathtools}
\usepackage{url}
\usepackage[margin=1in]{geometry}
\usepackage{fancyhdr}
\usepackage[all]{xy}

\numberwithin{equation}{section}

\title{Branched covers of elliptic curves and K\"ahler groups with exotic finiteness properties}
\author{Claudio Llosa Isenrich}
\address{Laboratoire de Math\'ematiques d'Orsay, Univ. Paris-Sud, CNRS, Universit\'e Paris-Saclay, 91405 Orsay, France}
\email{claudio.llosa-isenrich@math.u-psud.fr}
\thanks{This work was supported by a EPSRC Research Studentship and by the German National Academic Foundation}

\begin{document}

\newcommand{\QQ}{{\mathds Q}}
\newcommand{\RR}{{\mathds R}}
\newcommand{\NN}{{\mathds N}}
\newcommand{\ZZ}{{\mathds Z}}
\newcommand{\CC}{{\mathds C}}
\newcommand{\eps}{{\epsilon}}

\theoremstyle{plain}
\newtheorem{theorem}{Theorem}[section]
\newtheorem{conjecture}[theorem]{Conjecture}
\newtheorem{corollary}[theorem]{Corollary}
\newtheorem{lemma}[theorem]{Lemma}
\newtheorem{proposition}[theorem]{Proposition}
\newtheorem{question}{Question}

\theoremstyle{definition}
\newtheorem{remark}[theorem]{Remark}
\newtheorem*{acknowledgements*}{Acknowledgements}
\newtheorem{example}[theorem]{Example}
\newtheorem{definition}[theorem]{Definition}

\renewcommand{\proofname}{Proof}

\begin{abstract}
We construct K\"ahler groups with arbitrary finiteness properties by mapping products of closed Riemann surfaces holomorphically onto an elliptic curve: for each $r\geq 3$, we obtain large classes of K\"ahler groups that have classifying spaces with finite $(r-1)$-skeleton but do not have classifying spaces with finitely many $r$-cells. We describe invariants which distinguish many of these groups. Our construction is inspired by examples of Dimca, Papadima and Suciu. 
\end{abstract}

\maketitle

\section{Introduction}

A \textit{K\"ahler group} is a group which can be realised as fundamental group of a compact K\"ahler manifold. In particular, every K\"ahler group is finitely presented. A group $G$ is of \textit{finiteness type $\mathcal{F}_r$} if there is a $K(G,1)$ with finite $r$-skeleton; so type $\mathcal{F}_1$ is equivalent to finite generation and type $\mathcal{F}_2$ is equivalent to finite presentability.

In 2009, Dimca, Papadima and Suciu \cite{DimPapSuc-09-II} constructed the first examples of K\"ahler groups that are of finiteness type $\mathcal{F}_{r-1}$, but not of type $\mathcal{F}_r$, for all $r\geq 3$. In fact their groups are \textit{projective}, that is, fundamental groups of smooth projective manifolds. In particular, these groups, although torsion-free, can not have compact K\"ahler manifolds as classifying spaces. Our main goal here is to construct broader classes of such groups.

We will show how to construct large classes of new examples of K\"ahler groups with arbitrary finiteness properties from surjective maps of a direct product of surface groups onto an elliptic curve. More precisely, we will prove 

\begin{theorem}
Let $r\geq 3$ and let $g_1,\cdots, g_r \geq 2$. Let $E$ be an elliptic curve and, for $i=1,\cdots, r$, let $S_{g_i}$ be a closed connected surface of genus $g_i$ together with a branched covering $f_{g_i}:S_{g_i}\rightarrow E$. Assume that the map
\[
f=\sum_{i=1}^r f_{g_i}: S_{g_1}\times \cdots \times S_{g_r}\rightarrow E
\]
induces a surjective map on fundamental groups.

Then the generic fibre $H$ of $f$ is connected and $f$ induces a short exact sequence
\[
1\rightarrow \pi_1 H\rightarrow \pi_1 S_{g_1}\times \cdots \times \pi_1 S_{g_r}\rightarrow \pi_1 E\rightarrow 1
\] 
on fundamental groups. Moreover, the group $\pi_1H$ is a projective (and thus K\"ahler) group of type $\mathcal{F}_{r-1}$, but not of type $\mathcal{F}_r$.
 \label{thmFinProp}
\end{theorem}

Under the additional assumption that all maps $f_{g_i}$ are \textit{purely branched}, we will give the following complete classification of all K\"ahler groups that can be obtained using our construction. The notion of a purely branched covering will be defined in Definition \ref{defpurelybranched}. Roughly speaking a purely bran\-ched covering is a branched covering which is obtained from the base space by alone branching.

\begin{theorem}
 Let $E$ be an elliptic curve. Let $r,s\geq 3$, let $g_i,h_j\geq 2$, let $S_{g_i}$ be a closed Riemann surface of genus $g_i\geq 2$ and let $R_{h_i}$ be a closed Riemann surface of genus $h_i\geq 2$ for $1\leq i \leq r$ and $1\leq j \leq s$. Assume that there are purely branched $k_i$-fold holomorphic covering maps $p_i: S_{g_i}\rightarrow E$ and purely branched $l_i$-fold holomorphic covering maps $q_i: R_{h_i}\rightarrow E$. Define $p=\sum_{i=1}^r p_i : S_{g_1}\times \cdots \times S_{g_r} \rightarrow E$ and $q=\sum_{j=1}^s q_j: R_{h_1}\times \cdots \times R_{h_s}\rightarrow E$, and denote by $H_p$ and $H_q$ the smooth generic fibres of $p$, respectively $q$.
 
 Then the K\"ahler groups $\pi_1 H_p$ and $\pi_1 H_q$ are isomorphic if and only if $r=s$ and there is a permutation of the $R_{h_i}$ and $q_i$ such that $g_i=h_i$ and $k_i=l_i$ for $i=1,\cdots,r$.
 \label{thmClassfbKG}
\end{theorem}

The construction is inspired by the work of Dimca, Papadima and Suciu, but is more general. The construction given by Dimca, Papadima and Suciu is a topological construction and does not lead to an explicit finite presentation. In \cite{Llo-16}, Llosa Isenrich derived an explicit presentation for these examples, answering a question of Suciu. The original proof that Dimca, Papadima and Suciu's examples have arbitrary finiteness properties consists of an involved argument making use of characteristic varieties. Alternative proofs have been given since by Biswas, Mj and Pancholi \cite{BisMjPan-14} and by Suciu \cite{Suc-12}. We will give a shorter proof making use of the classification of the finiteness properties of subgroups of direct products of surface groups obtained by Bridson, Howie, Miller and Short \cite{BriHowMilSho-02}.

Biswas, Mj and Pancholi \cite{BisMjPan-14} suggested a more general approach for arbitrary irrational Lefschetz pencils over surfaces of positive genus with singularities of Morse type. Although it is not explicit in \cite{BisMjPan-14}, the class of examples constructed in our work can also be obtained as a consequence of their work. However, the techniques are quite different: their approach is based on topological Lefschetz fibrations which by definition have Morse type singularities and the result for non-Morse type singularities then follows by a deformation argument. They do not develop any techniques to distinguish between different examples (cf. Theorem \ref{thmClassfbKG} above).

As an example of our construction we will produce a class of K\"ahler groups which to us seem the most natural analogue in the K\"ahler setting of a specific subclass of the Bestvina--Brady groups \cite{BesBra-97}. Using our classification we will then show that these groups are not isomorphic to Dimca, Papadima and Suciu's groups. An interesting question that one can ask here is if one can generalise our construction further to produce natural K\"ahler analogues for all of the original Bestvina--Brady groups. Notice that Dimca, Papadima and Suciu showed that the only original Bestvina--Brady groups which are K\"ahler are free abelian groups of even rank \cite[Corollary 1.3]{DimPapSuc-08}.

The structure of this work is as follows: in Section \ref{secConnFib} we prove that the fibres of our maps are connected. In Section \ref{secBBgroups} we will prove an algebraic result which will allow us to deduce the finiteness properties of the groups arising from our construction. In Section \ref{secGenConst} we explain our general construction which we then apply in Section \ref{secSpecConst} to realise the algebraically described groups from Section \ref{secBBgroups} as K\"ahler groups. In Section \ref{secComparison} we explain how the question about the existence of isomorphisms between K\"ahler groups obtained using our construction reduces to a question in Linear Algebra. In Section \ref{secClassfb} we solve this question for maps arising from purely branched coverings and conclude that the groups in Section \ref{secSpecConst} are not isomorphic to Dimca, Papadima and Suciu's groups.

\begin{acknowledgements*}
I am very grateful to my advisor Martin Bridson for his generous support and advice while writing this paper and for the very helpful discussions that we had in our numerous meetings. I am also very greatful to Mahan Mj for the very helpful discussions which allowed me to simplify the proof of Lemma \ref{thmConnSurfGen} and for explaining to me how our examples can also be obtained using techniques from \cite{BisMjPan-14}. Moreover, I would like to thank the anonymous referee for helpful comments and suggestions that led to improvements to the exposition of this paper, and my office mates Alexander Betts and Giles Gardam for the many inspiring discussions about the content of this work.
\end{acknowledgements*}

\section{Connectedness of fibres}

\label{secConnFib}
In this section we prove that the fibres of the maps in Theorem \ref{thmFinProp} are connected.

\begin{lemma}
Let $r\geq 2$, let $E$ be an elliptic curve and let $S_{g_i}$ be a closed Riemann surface of genus $g_i\geq 2$. Let $f_{g_i}: S_{g_i}\rightarrow E$ be holomorphic branched covers of $E$. 

The map $f=\sum_{i=1}^r f_{g_i}:S_{g_1}\times \cdots \times S_{g_r} \rightarrow E$ has connected fibres if and only if it induces a surjective map on fundamental groups.
\label{thmConnSurfGen}
\end{lemma}

\begin{proof}
Every holomorphic map $h:X\rightarrow Y$ between compact complex manifolds $X$ and $Y$ with connected fibres induces a surjective map on fundamental groups, since it is a locally trivial fibration over the complement of a complex codimension one subvariety of $Y$. Hence, if $f$ has connected fibres then it induces a surjective map on fundamental groups.

 Assume now that $f$ induces a surjective map on fundamental groups. If $f$ does not have connected fibres then Stein factorisation yields a closed Riemann surface $S$ and holomorphic maps $\alpha: S\rightarrow E$ and $\beta: S_{g_1}\times \cdots \times S_{g_r} \rightarrow S$ such that $\alpha$ is finite-to-one and $\beta$ has connected fibres. Since holomorphic finite-to-one maps between closed Riemann surfaces are branched covering maps, it follows that $\alpha$ is a branched covering.
 
 Choose a base point $(p_1,\cdots,p_r)\in S_{g_1}\times \cdots \times S_{g_r}$ and denote by $\beta_i$ the restriction of $\beta$ to the $i$th factor $\left\{(p_1,\cdots,p_{i-1})\right\}\times S_{g_i}\times \left\{(p_{i+1},\cdots,p_r)\right\}$. Then there is $e_i\in E$ such that $\alpha\circ \beta_i = e_i + f_{g_i}$. Since $\beta_i$ is holomorphic, it is non-trivial and finite-to-one, and hence a finite-sheeted holomorphic branched covering map. It now follows that $\beta_{i\ast}(\pi_1 S_{g_i})\leq \pi_1 S$ is a finite index subgroup and therefore not cyclic for $i=1,\cdots,r$.
 
 It then follows from \cite[Lemma 7.1]{BisMjPan-14} that $S$ must itself be an elliptic curve. Since the argument is short we want to give it here: Choose $\gamma_1\in \pi_1 S_{g_1}$ and $\gamma_2\in \pi_1 S_{g_2}$ such that their images $\beta_1\circ \gamma_1$ and $\beta_2\circ \gamma_2$ do not lie in a common cyclic subgroup of $\pi_1 S$. Then $\beta_1\circ \gamma_1$ and $\beta_2\circ\gamma_2$ generate a $\ZZ^2$-subgroup of $\pi_1S$ and the only closed Riemann surfaces with $\ZZ^2$-subgroups are elliptic curves (e.g. \cite[Theorem 1]{Jac-70}).
 
 Using Euler characteristic, it follows that any branched covering map between 2-dimensional tori is an unramified covering map. By assumption the map $f=\alpha\circ \beta$ induces a surjective map on fundamental groups. Hence, the map $\alpha:S\rightarrow E$ is an unramified holomorphic covering which is surjective on fundamental groups and therefore $S=E$ and $\alpha$ is biholomorphic. In particular, the map $f$ has connected fibres.
\end{proof}

For a group $G$ and a subset $S\subset G$ denote by $\left\langle \left\langle S\right\rangle \right\rangle \leq G$ its normal closure in $G$. 

We introduce a special class of branched covering maps of tori. Let $Y$ be a torus of real dimension $k\geq 2$, let $X$ be a closed connected manifold of real dimension $k$ and let $D\subset Y$ be a real subvariety of codimension at least two (we will be interested in the case when $X$ and $Y$ are complex manifolds and $D$ is a complex subvariety). Let $f: X\rightarrow Y$  be a branched covering map with branching locus $D$, that is, $f^{-1}(D)$ is a nowhere dense set in $X$ mapping onto $D$ and the restriction $f: X\setminus f^{-1}(D) \rightarrow Y\setminus D$ is an unramified covering.

Assume that for a base point $z_0\in Y\setminus D$ there are simple closed loops $\mu_1,\dots,\mu_k,b_1,\dots,b_l: \left[0,1\right] \rightarrow Y\setminus D$ based at $z_0$ and points $p_1,\dots,p_l\in D$ with the following properties:
\begin{itemize}
\item $\pi_1 (Y\setminus D)=\left\langle \left[\mu_1\right],\dots,\left[\mu_k\right],\left[b_1\right],\dots,\left[b_l\right]\right\rangle$;
\item $\mu_i(\left[0,1\right])\cap \mu_j(\left[0,1\right]) = \left\{z_0\right\}$ for $i \neq j$;
\item the loops $\mu_1,\dots,\mu_k$ generates $\pi_1Y$; and
\item for any choice $U_1,\dots,U_l\subset Y$ of open neighbourhoods of $p_1,\dots,p_l$ there are paths $\delta_i : \left[0,1\right] \rightarrow Y\setminus D$ starting at $z_0$ and ending at a point in $U_i\setminus \left\{p_i\right\}$ and loops $\nu_i: \left[0,1\right]\rightarrow U_i\setminus \left\{p_i\right\}$ based at these points, such that for the concatenation $\beta_i=\delta_i\cdot \nu_i\cdot \delta_i^{-1}$, we have $\left[\beta_i\right] = \left[b_i\right]\in \pi_1 (Y\setminus D)$.
\end{itemize}

\begin{definition}
 We call the map $f$ \textit{purely branched} if there exist loops $\mu_1,\dots,\mu_k,b_1,\dots,b_l$ as above which satisfy the condition
\[
 \left\langle\left\langle \left[\mu_1\right],\dots,\left[\mu_k\right]\right\rangle\right\rangle  \leq f_{\ast} (\pi_1 (X\setminus f^{-1}(D)))\leq \pi_1 (Y\setminus D),
\]
 i.e. every lift to $X$ of each $\mu_i$ is a loop.
\label{defpurelybranched}
\end{definition}

Note that for $Y$ a 2-torus, $X$ a closed connected surface of genus $g\geq 1$ and $f:X\rightarrow Y $ a branched covering map with branching locus $D=\left\{p_1,\cdots,p_r\right\}$, the map $f$ is purely branched if and only if there are simple closed loops $\mu_1,\mu_2:\left[0,1\right]\rightarrow Y\setminus D$ which generate $\pi_1 Y$, intersect only in $\mu_1(0)=\mu_2(0)$ and have the property that every lift of $\mu_1$ and $\mu_2$ is a loop in $X$.

 We want to remark that for a branched covering map being purely branched is not the same as being surjective on fundamental groups, as Example \ref{ex:surjnotpb} at the end of Section \ref{secClassfb} will show.

\section{A K\"ahler analogue of Bestvina--Brady groups}

\label{secBBgroups}

A large class of groups with prescribed finiteness type are the Bestvina--Brady groups which were constructed using combinatorial Morse theory \cite{BesBra-97}. These groups arise as kernels of surjective maps from Right Angled Artin groups to the integers.

Given a finite simplicial graph $\Gamma$, denote by $V(\Gamma)$ its vertex set and by $E(\Gamma)$ its edge set. We define the \textit{Right Angled Artin group} (RAAG) $A_{\Gamma}$ for $\Gamma$ as the group with the finite presentation

\[
A_{\Gamma}=\left\langle V(\Gamma)\mid \left[v,w\right]\mbox{ if } vw\in E(\Gamma)\right\rangle.
\]

Associated to every such presentation of a RAAG is a natural homomorphism

\[
 \begin{array}{lrl}
 \phi_{\Gamma}:& A_{\Gamma} &\rightarrow \ZZ=\left\langle t\right\rangle \\
 & V(\Gamma)\ni v &\mapsto t\\
 \end{array}
\]

Associated to the RAAG $A_{\Gamma}$ we have the flag complex $X_{\Gamma}$ with 1-skeleton $\Gamma$, that is, the simplicial complex obtained from $\Gamma$ by requiring that whenever a finite set of vertices in $\Gamma$ is pairwise connected, there is a simplex in $X_\Gamma$ with these vertices. 

The \textit{Bestvina--Brady group} $BB_{\Gamma}$ associated to $\Gamma$ is defined as $BB_{\Gamma} = \mathrm{ker} \phi_{\Gamma}$. The finiteness properties of the Bestvina--Brady groups are completely understood:

\begin{theorem}[Bestvina--Brady Theorem, \cite{BesBra-97}]
Let $\Gamma$ be a finite simplicial graph. Then for every integer $n\geq 0$, the Bestvina--Brady group $BB_{\Gamma}$ is of type $\mathcal{F}_n$ if and only if the space $X_{\Gamma}$ is (n-1)-connected, that is, $\pi _i (X_{\Gamma}) = 1$ for all $0\leq i \leq n$.
\end{theorem}

Dimca, Papdima and Suciu showed that the only Bestvina--Brady groups which are K\"ahler are the obvious ones in the following sense 

\begin{theorem}[{\cite[Corollary 1.3]{DimPapSuc-08}}]
 Let $\Gamma$ be a finite simplicial graph. Then $BB_{\Gamma}$ is K\"ahler if and only if $BB_{\Gamma}$ is free abelian of even rank.
\end{theorem}

Free abelian groups are of type $\mathcal{F}_{\infty}$, since their classifying space is a torus. Hence, the Bestvina--Brady groups can not provide any examples of K\"ahler groups with arbitrary finiteness properties. However, Dimca, Papadima and Suciu \cite{DimPapSuc-09-II} observed that one can imitate the Bestvina--Brady construction to obtain K\"ahler groups with arbitrary finiteness properties. 

From a group theoretic point of view, their groups arise as kernels of maps from a direct product of surface groups to $\ZZ ^2$; We will describe the geometric construction behind these maps later in this paper. Let $\Lambda _{g_1},\cdots, \Lambda _{g_r}$, $g_1,\cdots,g_r\geq 2$ be surface groups with presentation

\[
 \Lambda_{g_i} = \left\langle a_1^i,\cdots,a_{g_i}^i,b_1^i,\cdots,b_{g_i}^i \mid \left[a_1^i,b_1^i\right]\cdots \left[a_{g_i}^i,b_{g_i}^i\right]\right\rangle
\]

Although they did not describe them as such, Dimca, Papadima and Suciu's groups are kernels of the surjective maps
\begin{equation}
\begin{array}{cclc}
 \phi_{g_1,\cdots,g_r}:& \Lambda_{g_1}\times \cdots\times \Lambda_{g_r} & \rightarrow &\ZZ ^2=\left\langle a,b \mid \left[a,b\right]\right\rangle\\
& a_1^i, a_2^i &\mapsto &a\\
& b_1^i, b_2^i &\mapsto & b\\
& a_3^i, \cdots, a_{g_i}^i &\mapsto & 0\\
& b_3^i, \cdots, b_{g_i}^i &\mapsto & 0,\\
\end{array}
\label{eqnDPSmap}
\end{equation}
as our explanation of their construction in Section \ref{secGenConst} will show, where $g_1,\cdots,g_r\geq 2$ and $r\geq 3$.

In particular, we will see that the maps constructed in Section \ref{secGenConst} satisfy all the conditions of Theorem \ref{thmFinProp}. This will provide us with a new proof that the groups  $\phi_{g_1,\cdots,g_r}$ are K\"ahler groups of finiteness type $\mathcal{F}_{r-1}$, but not of type $\mathcal{F}_r$. Thus:

\begin{theorem}
 If $g_1,\cdots,g_r\geq 2$, then the kernel of $\phi_{g_1,\cdots,g_r}$ is a K\"ahler group of finiteness type $\mathcal{F}_{r-1}$, but not of type $\mathcal{F}_r$.
\end{theorem}

The case when $g_1=\cdots = g_r=2$ seems to be a completely analogous surface group version of the Bestvina--Brady group corresponding to the direct product $F_2\times \cdots \times F_2$ of $r$ copies of the free group on two generators. 

More generally consider the direct product $\Lambda_{g_1}\times \cdots \times \Lambda_{g_r}$ for $g_1,\cdots,g_r\geq 2$. Then, in analogy to the map $\phi_{2,\cdots,2}$, define the homomorphism

\begin{equation}
\begin{array}{cclc}
 \psi_{g_1,\cdots,g_r}:& \Lambda_{g_1}\times \cdots\times \Lambda_{g_r} & \rightarrow &\ZZ ^2=\left\langle a,b\mid \left[a,b\right]\right\rangle\\
 &a_j^i &\mapsto & a \\
 &b_j^i&\mapsto & b\\
\end{array}
\label{eqnOurExamples}
\end{equation}

To prove that all groups arising as kernel of one of the $\psi_{g_1,\cdots,g_r}$ are not of type $\mathcal{F}_r$, we use a result by Bridson, Miller, Howie and Short \cite[Theorem B]{BriHowMilSho-02}

\begin{theorem}
Let $\Lambda _{g_1},\cdots, \Lambda _{g_n}$ be surface groups and let $G\leq \Lambda _{g_1}\times \cdots \times \Lambda _{g_n}$ be a subgroup of their direct product. Assume that each intersection $L_i= G\cap \Lambda _{g_i}$ is non-trivial and arrange the factors in such a way that $L_1,\cdots,L_r$ are not finitely generated and $L_{r+1},\cdots, L_n$ are finitely generated. 

If precisely $r\geq 1$ of the $L_i$ are not finitely generated, then $G$ is not of type $FP_r$.
\label{thmBHMS}
\end{theorem}

For finitely presented groups, being of type $FP_r$ is equivalent to being of type $\mathcal{F}_r$ \cite[p. 197]{Bro-82}. Since all the groups we consider will be finitely presented, we do not dwell on the meaning of the homological finiteness condition $FP_r$. As a consequence of Theorem \ref{thmBHMS} we can now prove

\begin{corollary}
 Let $\Lambda _{g_1},\cdots, \Lambda _{g_r}$ be surface groups with $g_i\geq 2$ and let $r\geq 3$. Let $k\geq 1$ and let $\nu_i: \Lambda _{g_i}\rightarrow \ZZ^k$ be non-trivial homomorphisms. Then the kernel of the map
 \[
 \nu=\nu_1+\cdots + \nu_r:\Lambda_{g_1}\times \cdots \times \Lambda_{g_r} \rightarrow \ZZ^k
 \]
 is not of type $\mathcal{F}_r$.
 \label{thmNotFr}
\end{corollary}
\begin{proof}
Let $G= \mathrm{ker}(\nu)<\Lambda_{g_1}\times \cdots \times \Lambda_{g_r}$. Let $L_i= G\cap \Lambda _{g_i}:= G\cap 1\times\cdots\times \Lambda_{g_i}\times\cdots\times 1$ and note that $L_i=\mathrm{ker}(\nu_i)$. Then $L_i$ is an infinite index normal subgroup of $\Lambda_{g_i}$. Infinite index normal subgroups of a surface group are infinitely generated free groups (e.g. \cite[Theorem 1]{Jac-70} and \cite[Theorem 3.1]{BriHow-07}). Hence, none of the $L_i$ are finitely generated and therefore $G$ is not of type $FP_r$, by Theorem \ref{thmBHMS}, and hence not of type $\mathcal{F}_r$. 
\end{proof}

In particular this gives an alternative proof of the finiteness properties of the K\"ahler groups constructed by Dimca, Papadima and Suciu. In fact, we immediately obtain

\begin{corollary}
 For all $r\geq 3$ and $g_1,\cdots,g_r\geq 2$, the groups\\ $\mathrm{ker}(\phi_{g_1,\cdots,g_r})$ and  $\mathrm{ker}(\psi_{g_1,\cdots,g_r})$ are not of type $\mathcal{F}_r$.
\end{corollary}

\begin{remark}
Under the additional condition that the maps $f_i : \Lambda_i \rightarrow \ZZ^k$ are surjective all groups that arise in this way are group theoretic fibre products over $\ZZ^k$; one can construct explicit finite presentations for them using the same methods as in \cite{Llo-16}.
\end{remark}

\section{Constructing large classes of examples}
\label{secGenConst}

In this section we will show how to use our results to prove Theorem \ref{thmFinProp}. Let $E$ be a closed Riemann surface of positive genus and let $X$ be a closed connected complex analytic manifold of dimension $r>1$. An \textit{irrational pencil} $h: X\rightarrow E$ is a surjective holomorphic map such that the smooth generic fiber $H$ is connected. 

For $M$, $N$ complex manifolds, we say that a surjective holomorphic map  $f:M\rightarrow N$ has \textit{isolated singularities} if, for every $y\in N$, the set of singular points of $f$ in $f^{-1}(y)$ is discrete. In the case of an irrational pencil $h:X\rightarrow E$ this is equivalent to the set of singular points of $f$ being finite.

The following result is due to Dimca, Papadima and Suciu

\begin{theorem}[\cite{DimPapSuc-09-II}, Theorem C]
 Let $h: X\rightarrow E$ be an irrational pencil. Suppose that $h$ has only isolated singularities. Then the following hold:
 \begin{enumerate}
 \item The inclusion $H\hookrightarrow X$ induces isomorphisms $\pi _i (H)\cong \pi_i (X)$ for $2\leq i \leq r-2$;
 \item $h$ induces a short exact sequence $ 1\rightarrow \pi_1 H \rightarrow \pi_1 X\rightarrow \pi_1 E\rightarrow 1$.
 \end{enumerate}
 \label{thmDPSC}
\end{theorem}

A \textit{Stein manifold} is a complex manifold that embeds biholomorphically as a closed submanifold in some affine complex space $\CC^r$. Theorem \ref{thmDPSC} allows us to prove

\begin{theorem}
 Let $r\geq 3$ and let $g_1,\cdots, g_r \geq 2$. Let $E$ be an elliptic curve and, for $i=1,\cdots, r$, let $S_{g_i}$ be a closed connected surface of genus $g_i$ together with a branched covering $f_{g_i}:S_{g_i}\rightarrow E$. Assume further that the map
\[
f=\sum_{i=1}^r f_{g_i}: S_{g_1}\times \cdots \times S_{g_r}\rightarrow E
\]
induces a surjective map on fundamental groups.

Then the generic fibre $H$ of the map $f$ is a connected $(r-1)$-dimensional smooth projective variety such that

\begin{enumerate}
 \item The homotopy groups $\pi_i H$ are trivial for $2\leq i \leq r-2$ and $\pi_{r-1} H$ is nontrivial;
 \item The universal cover $\widetilde{H}$ of $H$ is a Stein manifold;
 \item The fundamental group $\pi_1 H$ is a projective (and thus K\"ahler) group of finiteness type $\mathcal{F}_{r-1}$, but not of finiteness type $\mathcal{F}_r$;
 \item The fundamental group $\pi_1 H$ is not commensurable (up to finite kernels) to any group having a classifying space of finite type;
 \item The map $f$ induces a short exact sequence $$1\rightarrow \pi_1 H \rightarrow \pi_1 S_{g_1}\times \cdots \times \pi_1 S_{g_r}\rightarrow \pi_1 E\rightarrow 1$$ on fundamental groups.
\end{enumerate}
\label{thmFinPropGen}
\end{theorem}

\begin{proof}
It is well-known that there is a unique complex structure on $S_{g_i}$ with respect to which the map $f_{g_i}$ is holomorphic, since $f_{g_i}$ is a finite-sheeted branched covering map. In particular, $f_{g_i}$ has critical points the finite preimage $C_i=f_{g_i}^{-1}(D_i)$, where $D_i\subset E$ is the finite set of branching points of $f_{g_i}$. 

This equips $X=S_{g_1}\times \cdots \times S_{g_r}$ with a projective structure with respect to which $f: X=S_{g_1}\times \cdots \times S_{g_r}\rightarrow E$ is a surjective holomorphic map. The set of singular points of $f$ is then the set of points $(x_1,\cdots, x_r)\in S_{g_1}\times \cdots \times S_{g_r}$ such that $0=\mathrm{d}f(x_1,\cdots, x_r) = \left(\mathrm{d}f_{g_1}(x_1),\cdots, \mathrm{d}f_{g_r}(x_r)\right)$. It follows that the set of singular points of $f$ is the finite set $C_1\times \cdots \times C_r$. In particular, $f$ has isolated singularities.

By assumption all of the $f_{g_i}$ are branched covering maps and $f$ is surjective on fundamental groups. Thus, by Lemma \ref{thmConnSurfGen}, the map $f$ has connected fibers.

It follows that $f$ is an irrational pencil with isolated singularities. Hence, by Theorem \ref{thmDPSC}, we obtain that $f$ induces a short exact sequence
\[
 1\rightarrow \pi_1 H \rightarrow \pi_1 X \xrightarrow[]{f_{\ast}} \pi_1 E \rightarrow 1
\]
on fundamental groups proving assertion (5). Furthermore, we obtain that $\pi_i H\cong \pi_i X\cong 0$, for $2\leq i\leq r-2$, where the last equality follows since $X$ is a $K(\pi_1 X,1)$. This implies the first part of assertion (1).

The group $\pi_1 H$ is projective, since the generic smooth fibre $H$ of $f$ is a complex submanifold of the compact projective manifold $X$.
 
Because $\pi_i H=0$ for $2\leq i\leq r-2$, we obtain a $K(\pi_1H,1)$ from $H$ by attaching cells of dimension $\geq r$. Since $H$ is a compact complex manifold, it follows that $H$ has a finite cell structure. Thus, the group $\mathrm{ker}(f_{\ast})= \pi_1 H$ is of finiteness type $\mathcal{F}_{r-1}$ and, by Corollary \ref{thmNotFr}, it is not of type $\mathcal{F}_r$. This implies assertion (3) and the second part of assertion (1), since if $\pi_{r-1}H$ were trivial we could construct a $K(G,1)$ with finite $r$-skeleton.

Assertion (4) is an immediate consequence of the well-known \cite[Proposition 2.7]{DimPapSuc-09-II}.

Assertion (1) follows similarly as in the proof of \cite[Theorem A]{DimPapSuc-09-II}. Namely, the universal covering $\widetilde{X}$, $q:\widetilde{X}\rightarrow X$ of $X=S_{g_1}\times \cdots \times S_{g_r}$ is a contractible Stein manifold and the pair $(X,H)$ is $(r-1)$-connected by Theorem \ref{thmDPSC}. Hence, the preimage $q^{-1}(H)$ of $H$ in $\widetilde{X}$ is a closed complex submanifold of the Stein manifold $\widetilde{X}$ which is biholomorphic to the universal covering $\widetilde{H}$ of $H$. Thus, $\widetilde{H}$ is Stein.
\end{proof}

Theorem \ref{thmFinProp} is now a direct consequence of Theorem \ref{thmFinPropGen}. Note that Theorem \ref{thmFinPropGen} is a generalisation of \cite[Theorem A]{DimPapSuc-09-II}.

\section{Constructing a Bestvina--Brady type class of examples}
\label{secSpecConst}

We will now proceed to explain how one can realise the maps $\phi_{g_1,\cdots,g_r}$ in \eqref{eqnDPSmap} and $\psi_{g_1,\cdots,g_r}$ in \eqref{eqnOurExamples} geometrically. More precisely, they will arise as the induced maps on fundamental groups of maps satisfying the conditions of Theorem \ref{thmFinProp}. 

The map realising $\phi_{g_1,\cdots,g_r}$ is the map used for the original construction of K\"ahler groups with arbitrary finiteness properties by Dimca, Papadima and Suciu. Since our construction imitates the construction by Dimca, Papadima and Suciu, we want to give a brief description of their construction first.

Let $E$ be an elliptic curve, that is, a complex torus of dimension one, let $g\geq 2$, and let $B=\left\{b_1,\cdots, b_{2g-2}\right\} \subset E$ be a finite subset of even size. Choose a set of generators $\alpha,\beta,\gamma_1,\cdots, \gamma_{2g-2}$ of the first homology group $H_1(E\setminus B, \ZZ)$ such that $\gamma _i$ is the boundary of a small disc centred at $b_i$, $i=1,\cdots, 2g-2$, and $\alpha, \beta$ are generators of $\pi_1 E$ which have simple closed representatives intersecting in a single point. 

Then the map $H_1(E\setminus B,\ZZ)\rightarrow \ZZ/2\ZZ$ defined by $\gamma_i\mapsto 1$ and $\alpha,\beta \mapsto 0$ induces a 2-fold normal covering of $\pi_1 (E\setminus B)$ which extends continuously to a 2-fold branched covering $f_g: S_g\rightarrow E$ from a topological surface $S_g$ of genus $g\geq 2$ onto $E$. It is well-known there is a unique complex structure on $S_g$ such that the map $f_g$ is holomorphic.

By looking at a suitable concrete realisation of the map $f_g$ (see \cite{Llo-16} for more details), it is not hard to see that it induces the map
\[
\begin{array}{cccc}
f_{g\ast} : \Lambda_g=\pi_1S_g &\rightarrow & \pi_1E =\ZZ^2&\\
a_1,a_2&\mapsto &a&\\
b_1,b_2 &\mapsto &b&\\
a_i,b_i&\mapsto &0 &\mbox{   }\forall i =3,\cdots,g\\
\end{array}
\] 
with respect to presentations $\Lambda _g =\left\langle a_1,b_1,\cdots,a_g,b_g\mid \left[a_1,b_1\right]\cdots \left[a_g,b_g\right]\right\rangle$ and $\pi_1 E = \left\langle a,b\mid \left[a,b\right]\right\rangle$ of the fundamental groups.

Then for $g_1,\cdots,g_r\geq 2$ and $r\geq 3$, use addition in the elliptic curve to define the map $f=\sum _{i=1}^r f_{g_i}: S_{g_1}\times \cdots \times S_{g_r}\rightarrow E$.

The maps $f_{g_i}$ are branched coverings inducing surjective maps on fundamental groups. They restrict to unramified normal coverings of $E\setminus B_i$ which correspond to kernels of homomorphisms mapping a set of generators $\alpha,\beta:\left[0,1\right]\rightarrow E\setminus B_i$ of $\pi_1E$ to zero. It follows that $f_{g_i}$ is purely branched for $i=1,\cdots,r$. In particular they are surjective on fundamental groups.

Hence, by Theorem \ref{thmFinProp}, the kernel of the induced map $f_{\ast}: \Lambda_{g_1}\times \cdots \times \Lambda_{g_r}\rightarrow \ZZ^2$ is projective (and thus K\"ahler) of finiteness type $\mathcal{F}_{r-1}$, but not $\mathcal{F}_r$. Observe that the map $f_{\ast}$ is indeed identical with the map $\phi_{g_1,\cdots,g_r}$ defined in \eqref{eqnDPSmap}.

We will imitate the construction of these groups in order to produce holomorphic maps $\overline{f}_{h_1},\cdots, \overline{f}_{h_r}:S_{h_i}\rightarrow E$ for all $h_1,\cdots,h_r \geq 2$ and $r\geq 3$ such that $\overline{f}=\sum_{i=1}^r \overline{f}_{h_i}$ realises the map $\psi_{h_1,\cdots, h_r}$ defined in \eqref{eqnOurExamples}. We will present two different constructions, each of which has advantages.

\vspace{1cm}

\textbf{Construction 1}: This is the more natural construction. It has the advantage that the maps $\overline{f}_{h_i}$ are normal branched coverings, but it comes at the cost that the singularities of $\overline{f}$ are not quadratic and therefore $\overline{f}$ is not a Morse function.

As above, let $E$ be an elliptic curve and let $B=\left\{d_1,d_2\right\}\subset E$ be two arbitrary points. Let $\alpha, \beta, \gamma _1, \gamma _2$ be generators for the homology group $H_1(E\setminus B, \ZZ)$ of the form described in the paragraph preceding Definition \ref{defpurelybranched}, $\alpha, \beta$ generate $\pi_1 E$ and have simple closed representatives in $E\setminus B$ intersecting positively with respect to the orientation induced by the complex structure on $E$, and $\gamma_1, \gamma _2$ are the positively oriented boundary loops of small discs around $b_1$, respectively $b_2$. For $h\geq 2$, the surjective homomorphism $H_1(E\setminus B,\ZZ)\rightarrow \ZZ/ h\ZZ$ defined by $\alpha, \beta \mapsto 0$, $\gamma_1\mapsto 1$, $\gamma_2\mapsto -1$, defines a $h$-fold normal branched covering $\overline{f}_{h}: S_h \rightarrow E$ from a topological surface of genus $h$ with branching locus $B$. We may assume that, after connecting the generators of $H_1(E\setminus B,\ZZ)$ to a base point, the fundamental group of $E\setminus B$ is 
\[
\pi_1 \left(E\setminus B\right) = \left\langle \alpha,\beta,\gamma_1,\gamma_2\mid \left[\alpha,\beta\right]\gamma_1\gamma_2\right\rangle.
\]

It is well-known that there is a unique complex structure on $S_h$ such that the map $\overline{f}_h$ is holomorphic. Denote by {\small $C=\left\{c_1=\overline{f}_h^{-1}(d_1), c_2=\overline{f}_h^{-1}(d_2)\right\} \subset S_h$} the set of critical points of $\overline{f}_h$. 

In analogy to the DPS groups, we define the map $\overline{f}$ using the additive structure on $E$,
\[
\overline{f}=\sum_{i=1}^r \overline{f}_{h_i}: S_{h_1}\times \cdots \times S_{h_r} \rightarrow E
\]
for all $r\geq 3$ and $h_1,\cdots, h_r\geq 2$.

The maps $\overline{f}_{h_i}$ are branched coverings induced by the surjective composition of homomorphisms $\pi_1 (E\setminus B)\rightarrow H_1(E\setminus B,\ZZ)\rightarrow \ZZ/h_i\ZZ$ defined by $\alpha, \beta \mapsto 0$, $\gamma_1\mapsto 1$, $\gamma_2\mapsto -1$. In particular, all of the $\overline{f}_{h_i}$ are purely branched, since $\alpha$ and $\beta$ are elements of the kernel of this homomorphism, which is a normal subgroup of $\pi_1(E\setminus B)$. It follows that all of the $\overline{f}_{h_i}$ are surjective on fundamental groups.

Hence, Theorem \ref{thmFinPropGen} can be applied to $\overline{f}$. This implies that the fundamental group $\pi_1H$ of the generic fibre $H$ of $\overline{f}$ is a projective (and thus K\"ahler) group of finiteness type $\mathcal{F}_{r-1}$, but not of finiteness type $\mathcal{F}_r$.

Since all of the $\overline{f}_{h_i}$ are purely branched, any lift of a generator $\alpha,\beta:\left[0,1\right]\rightarrow E\setminus B_i$ of $\pi_1E$ to $S_{h_i}$ is a loop. Choose a fundamental domain $F\subset S_{h_i}$ for the $\ZZ/ h_i \ZZ$-action such that the images of $\alpha$ and $\beta$ are contained in the image $\overline{f}_{h_i}(U)$ of an open subset $U\subset F$ on which $\overline{f}_{h_i}$ restricts to a homeomorphism. 

Denote the $h_i$ lifts of $\alpha$ by $a_1^{(i)},\cdots, a_{h_i}^{(i)}$ and the $h_i$ lifts of $\beta$ by $b_1^{(i)},\cdots b_{h_i}^{(i)}$, where we choose lifts so that $a_j^{(i)}$ and $b_j^{(i)}$ are in the interior of the fundamental domain $j\cdot F$ for $j \in \ZZ/ h_i \ZZ$. In particular, the loops $a_1^{(i)},b_1^{(i)},\cdots,a_{h_i}^{(i)},b_{h_i}^{(i)}$ form a standard symplectic basis for the (symplectic) intersection form on $H_1(S_{h_i},\ZZ)$.

It is then well-known that we can find generators $\alpha_1^{(i)},\beta_1^{(i)},\cdots \alpha_{h_i}^{(i)},\beta_{h_i}^{(i)}$ of $\pi_1 S_{h_i}$ such that the abelianisation is given by $\alpha_j^{(i)}\mapsto a_j^{(i)}$, $\beta_j^{(i)}\mapsto b_j^{(i)}$ and
\[
 \pi_1S_{h_i} = \left\langle \alpha_1^{(i)},\beta_1^{(i)},\cdots \alpha_{h_i}^{(i)},\beta_{h_i}^{(i)}\mid \left[\alpha_1^{(i)},\beta_1^{(i)}\right] \cdots \left[\alpha_{h_i}^{(i)},\beta_{h_i}^{(i)}\right]\right\rangle.
\]
This is for instance an easy consequence of Theorem \ref{thmMCGSp}.

\begin{figure}[ht]
\centering{
\includegraphics[width=10cm,keepaspectratio]{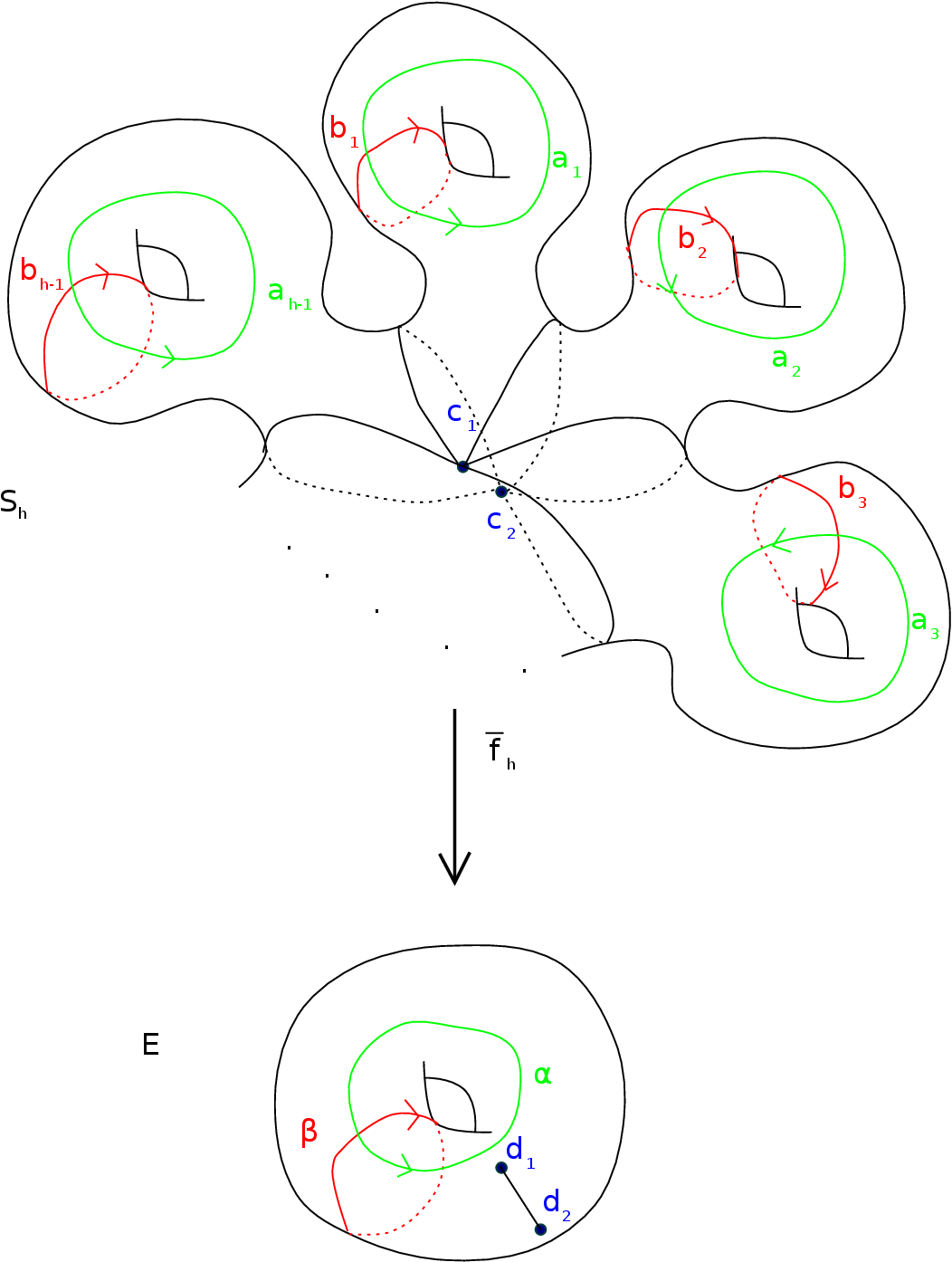}
\caption{The $h$-fold branched normal covering $\overline{f}_h$ of $E$ in Construction 1}
\label{fignfoldbranched}
}
\end{figure}

With respect to this presentation, the map on fundamental groups induced by $\overline{f}_{h_i}$ is given by
\[
 \begin{array}{cccc}
  \overline{f}_{h_i\ast}:& \pi_1 S_{h_i} &\rightarrow & \pi_1 E\\
  &\alpha_j^{(i)} &\mapsto & \alpha\\
  &\beta _j^{(i)} &\mapsto & \beta\\
 \end{array}
\]
For an illustration of the map $\overline{f}_{h}$ and the generators $\alpha_j,\beta_{j}$, see Figure \ref{fignfoldbranched}.

As a direct consequence, we obtain that $\overline{f}$ induces the map
\[
  \begin{array}{cccc}
  \overline{f}_{\ast}:& \pi_1 S_{h_1}\times \cdots \times \pi_1 S_{h_r} &\rightarrow & \pi_1 E\\
  &\alpha_j^{(i)} &\mapsto & \alpha\\
  &\beta _j^{(i)} &\mapsto & \beta\\
 \end{array}
\]
on fundamental groups. Thus, the induced map $\overline{f}_{\ast}$ on fundamental groups is indeed the map $\psi_{h_1,\cdots, h_r}$ in \eqref{eqnOurExamples}.

\vspace{1cm}

\textbf{Construction 2:} We will now give an alternative construction which realises $\psi_{h_1,\cdots,h_r}$ as the fundamental group of the generic fibre of a fibration over an elliptic curve with Morse type singularities only. This is at the expense of the maps being regular branched coverings rather than normal branched coverings. 

Let $E$ be an elliptic curve, let $h\geq 2$, let $d_{1,1},d_{1,2},d_{2,1},\cdots , d_{h-1,1},d_{h-1,2}$ be $2(h-1)$ points in $E$ and let $s_1,\cdots, s_{h-1}: \left[0,1\right]\rightarrow E$ be simple, pairwise non-intersecting paths with starting point $s_i(0)=d_{i,1}$ and endpoint $s_i(1)=d_{i,2}$ for $i=1,\cdots,h-1$. Take $h$ copies $E_0$, $E_1,\cdots, E_{h-1}$ of $E$, cut $E_0$ open along all of the paths $s_i$ and cut $E_i$ open along the path $s_i$ for $1\leq i \leq h-1$. This produces surfaces $F_0,\cdots, F_{h-1}$ with boundary.

Glue the boundary of $F_i$ to the boundary component of $F_0$ corresponding to the cut produced by the path $s_i$ where we identify opposite edges with respect to the identity homeomorphism $E_0\rightarrow E_i$ for $i=1,\cdots, h-1$. This yields a closed genus $h$ surface $R_h$ together with a $(h-1)$-fold branched covering map $f'_h: R_h\rightarrow E$ with critical points $c_{1,1}$, $c_{1,2}$, $c_{2,1}$,$\cdots$, $c_{h-1,1}$, $c_{h-1,2}$ $\in R_h$, $f'(c_{i,j})=d_{i,j}$ of order two. Endow $R_h$ with the unique complex structure that makes the map $f'_h$ holomorphic.

\begin{figure}[ht]
\centering{
\includegraphics[width=10cm,keepaspectratio]{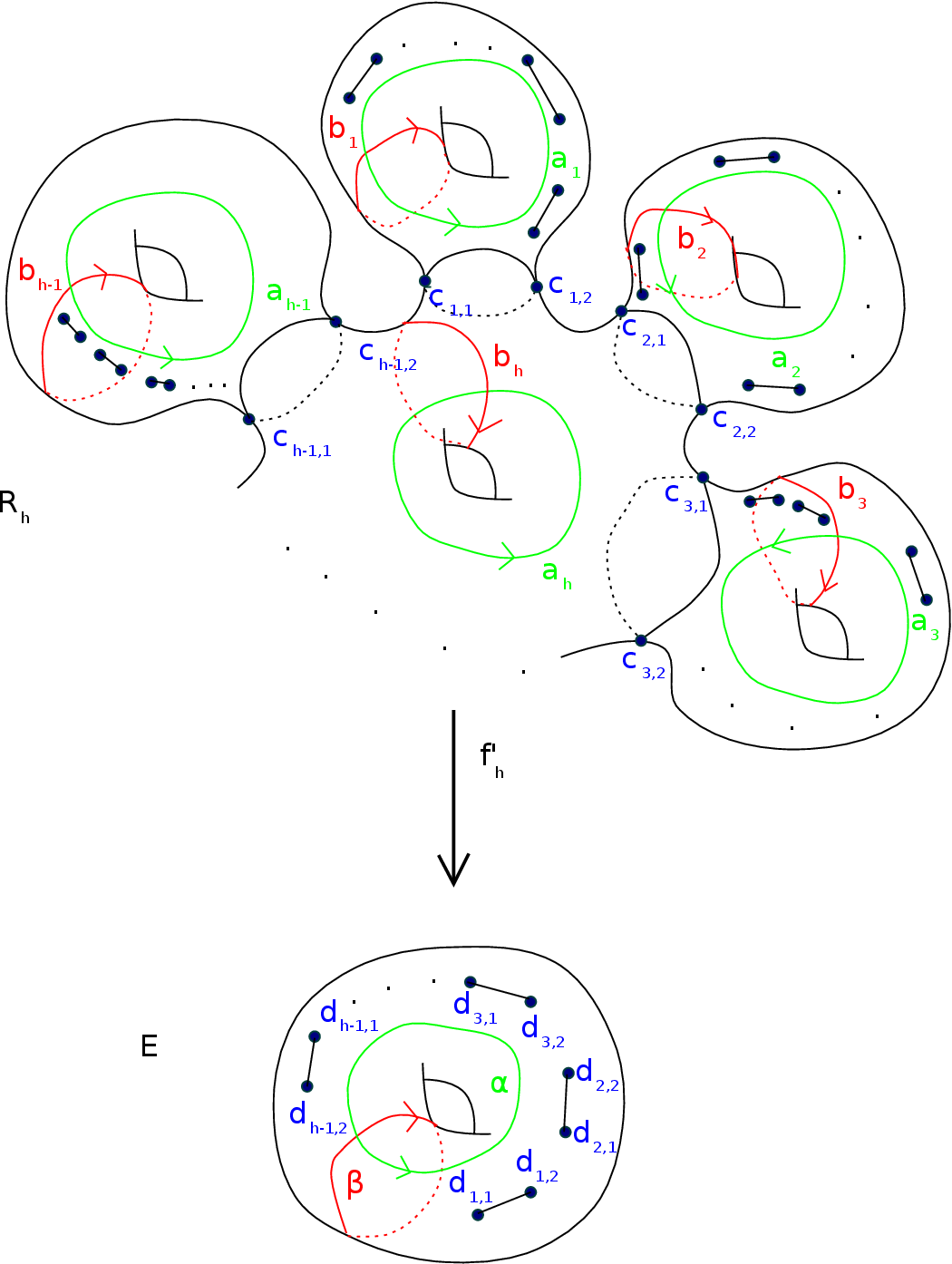}
\caption{The $h$-fold branched covering $f'_h$ of $E$ with Morse type singularities in Construction 2}
\label{figMorsetype}
}
\end{figure}

It is clear that the map $f'_h$ is purely branched and surjective on fundamental groups: in analogy to Construction 1 we find standard generating sets $\alpha,\beta$ of $\pi_1 E$ and $\alpha_1,\cdots,\alpha_h,\beta_1,\cdots,\beta_h$ of $\pi_1 R_h$ with respect to which the induced map on fundamental groups is given by $\alpha_i\mapsto \alpha$ and $\beta_i\mapsto \beta$. See Figure \ref{figMorsetype} for an illustration of the map $f'_h$.

For $h_1,\cdots,h_r\geq 2$, $r\geq 3$, the holomorphic map $f'= \sum _{i=1}^r f'_{h_i}: R_{h_1}\times \cdots \times R_{h_r}\to E$ induces the map $\psi_{h_1,\cdots,h_r}$ on fundamental groups. The map $f'$ has isolated singularities and connected fibres and in fact we can see, by considering local coordinates around the singular points, that all singularities of $f'$ are of Morse type.

\section{Reducing the isomorphism type of our examples to Linear Algebra}
\label{secComparison}

We will now show that our groups are not isomorphic to Dimca, Papadima and Suciu's groups and thus provide genuinely new examples rather than being their examples in disguised form.

As before, let $\Lambda_{g}$ be the fundamental group of a closed orientable surface of genus $g$. For $r,k \geq 1$, consider epimorphisms $\phi_{g_i}: \Lambda_{g_i}\rightarrow \ZZ^k$ and $\psi_{h_i}:\Lambda_{h_i}\rightarrow \ZZ^k$, where $g_i, h_i \geq 2$ and $1\leq i\leq r$. Recall that $\mathrm{ker}(\phi_{g_i})$ and $\mathrm{ker}(\psi_{h_i})$ are infinitely generated free groups, since they are infinite index normal subgroups of surface groups.

Define maps $\phi_{g_1,\cdots,g_r}=\phi_{g_1}+\cdots + \phi_{g_r}:\Lambda_{g_1}\times \cdots \times \Lambda_{g_r}\rightarrow \ZZ^k$ and $\psi _{h_1,\cdots,h_r}=\psi_{h_1}+\cdots + \psi_{h_r}:\Lambda_{h_1}\times\cdots\times \Lambda_{h_r}\rightarrow \ZZ^k$ and let
\[
L_i:= \Lambda_{g_i}\cap \mathrm{ker}(\phi_{g_1,\cdots,g_r}) = \mathrm{ker}(\phi_{g_i}),
\]
\[
K_i:= \Lambda_{h_i}\cap \mathrm{ker}(\psi_{h_1,\cdots,h_r}) = \mathrm{ker}(\psi_{h_i}).
\]

Then the following Lemma is a special case of \cite[Theorem C (3),(4)]{BriHowMilSho-13}

\begin{lemma}
Every isomorphism
\[
\theta : \mathrm{ker}(\phi_{g_1,\cdots,g_r})\rightarrow \mathrm{ker}(\psi_{h_1,\cdots,h_r})
\]
satisfies $\theta(L_i)=K_i$ up to reordering of the factors. In particular, we obtain that $\theta$ restricts to an isomorphism $L_1 \times \cdots \times L_r \cong K_1\times \cdots \times K_r$.

Furthermore, with the same reordering of factors, we have $\theta((\Lambda_{g_{i_1}}\times \cdots \Lambda_{g_{i_k}})\cap \mathrm{ker}(\phi_{g_1,\cdots,g_r}))= (\Lambda_{h_{i_1}}\times \cdots \Lambda_{h_{i_k}})\cap \mathrm{ker}(\psi_{h_1,\cdots,h_r})$ for $1\leq k\leq r$ and $1\leq i_1 <\cdots < i_k\leq r$.
\label{lemCharacteristic}
\end{lemma}

\begin{remark}
It also follows from \cite[Theorem C (3),(4)]{BriHowMilSho-13} that if we have an isomorphism between direct products of $r$ surface groups $\Lambda_{g_1}\times\cdots\times \Lambda_{g_r}$ and $\Lambda_{h_1}\times \cdots\times \Lambda_{h_r}$, then after reordering of the factors it is induced by isomorphisms $\Lambda _{g_i}\cong \Lambda_{h_i}$ and, in particular, $g_i=h_i$ for $i=1,\cdots,r$.
\label{rmkIsomProd}
\end{remark}

We obtain the following consequence

\begin{proposition}
There is an isomorphism of the short exact sequences
\[
1\rightarrow \mathrm{ker}(\phi_{g_1,\cdots,g_r})\rightarrow \Lambda_{g_1}\times \cdots\times \Lambda_{g_r} \xrightarrow[]{\phi_{g_1,\cdots,g_r}} \ZZ^k\rightarrow 1
\]
and
\[
1\rightarrow \mathrm{ker}(\psi_{h_1,\cdots,h_r})\rightarrow \Lambda_{h_1}\times\cdots\times \Lambda_{h_r}\xrightarrow[]{\psi_{h_1,\cdots,h_r}} \ZZ^k\rightarrow 1
\]
if and only if (up to reordering factors) $g_1=h_1, \cdots, g_r=h_r$ and there are isomorphisms of the short exact sequences
\[
1\rightarrow \mathrm{ker}(\phi_{g_i})\rightarrow \Lambda_{g_i}\xrightarrow[]{\phi_{g_i}} \ZZ^{k}\rightarrow 1
\]
and
\[
1\rightarrow \mathrm{ker}(\psi_{h_i})\rightarrow \Lambda_{h_i}\xrightarrow[]{\psi_{h_i}} \ZZ^{k}\rightarrow 1
\]
for all $i=1,\cdots, r$ such that the isomorphism $A:\ZZ^k\rightarrow \ZZ^k$ is independent of $i$.
\label{thmIsomSESeq}
\end{proposition}

\begin{proof}
The if direction follows immediately by taking the Cartesian product of the automorphisms $\theta_i:\Lambda_{g_i}\rightarrow \Lambda_{g_i}$ to be the automorphism of $\Lambda_{g_1}\times\cdots \times \Lambda_{g_r}$ and the identity map to be the automorphism of $\ZZ^k$.

For the only if direction we use that, by Remark \ref{rmkIsomProd}, the isomorphism $\Lambda_{g_1}\times\cdots\Lambda_{g_r}\rightarrow \Lambda_{h_1}\times\cdots\times \Lambda_{h_r}$ is realised by a direct product of isomorphisms $\theta_i: \Lambda_{g_i}\rightarrow \Lambda_{h_i}$, after possibly reordering factors. Restricting to factors then implies that the automorphism $\theta_i$ of $\Lambda_{g_i}$ and the identity on $\ZZ^k$ induce an isomorphism of short exact sequences for $i=1,\cdots,r$.
\end{proof}

\begin{proposition}
 Let $r\geq 2$. The groups $H_1= \mathrm{ker}(\phi_{g_1,\cdots,g_r})$ and $H_2=\mathrm{ker}(\psi_{h_1,\cdots,h_s})$ are isomorphic if and only if there is an isomorphism of the short exact sequences 
 \[
1\rightarrow \mathrm{ker}(\phi_{g_1,\cdots,g_r})\rightarrow \Lambda_{g_1}\times \cdots\times \Lambda_{g_r} \xrightarrow[]{\phi_{g_1,\cdots,g_r}} \ZZ^k\rightarrow 1
\]
and
\[
1\rightarrow \mathrm{ker}(\psi_{h_1,\cdots,h_s})\rightarrow \Lambda_{h_1}\times\cdots\times \Lambda_{h_s}\xrightarrow[]{\psi_{h_1,\cdots,h_s}} \ZZ^k\rightarrow 1.
\]
\label{thmAbstractIsom}
\end{proposition}
\begin{proof}
 Let $\theta : H_1\rightarrow H_2$ be an abstract isomorphism of groups and let $L_i\leq H_1$, $K_i\leq H_2$ be as above. By Lemma \ref{lemCharacteristic}, we may assume that after reordering factors $\theta(L_i)\leq K_i$ and that $\theta(M_1)=M_2$ for $M_1=H_1 \cap (1\times \Lambda_{g_2}\times \cdots \times \Lambda_{g_r})$, $M_2=H_2\cap (1\times \Lambda_{h_2}\times \cdots \times \Lambda_{h_s})$.
 
Since $\phi_{g_i}$, $\psi_{h_j}$ are surjective for all $1\leq i\leq r$, $1\leq j\leq s$, we obtain that 
\[
\begin{array}{ll}
H_1/M_1 \cong \Lambda_{g_1}, \mbox{        }& H_1/L_1\cong \Lambda_{g_2}\times \cdots \times, \Lambda_{g_r}\\
H_2/M_2 \cong \Lambda_{h_1}, & H_2/K_1 \cong \Lambda _{h_2}\times \cdots \times \Lambda_{h_s},
\end{array}
\]
where the isomorphisms are induced by the projection maps.

In particular, the map $\theta$ induces an isomorphism of short exact sequences
{\fontsize{8}{15}
\[
\xymatrix{1 \ar[r]  & K_1 \ar[r] \ar[d]^{\cong}_{\theta}  & H_1/M_1\times H_1/L_1 \ar[d]^{\cong}_{\theta}\ar@{}[r]\cong &*+[l]{ \Lambda_{g_1}\times \cdots \times \Lambda_{g_r}} \ar[r] & \ZZ^k \ar[r]\ar[d]^{\cong}_{\overline{\theta}} & 1\\ 
	  1 \ar[r] & K_2 \ar[r] & H_2/M_2\times H_2/K_1\cong &*+[l]{ \Lambda_{h_1}\times \cdots \times \Lambda_{h_s}} \ar[r]  & \ZZ^k \ar[r]& 1 \\}
\]}
proving the only if direction. The if direction is trivial. 
\end{proof}

Note that Proposition \ref{thmAbstractIsom} is well-known (see for instance \cite[Section 7.5]{BriHowMilSho-13}). Observe that the result does not hold for $r=1$: every infinite index normal subgroup of a hyperbolic surface group is a non-finitely generated free group; thus we get isomorphisms between the kernels of any two epimorphisms $\phi_{g_1}:\Lambda_{g_1}\rightarrow \ZZ^k$ and $\psi_{h_1}:\Lambda _{h_1}\rightarrow \ZZ^k$, with $g_1,h_1\geq 2$, while clearly the corresponding short exact sequences are not isomorphic for $g_1\neq h_1$.

It follows that in order to understand abstract isomorphisms of the kernels of maps of the form $\phi_{g_1,\cdots,g_r}$ and $\psi_{h_1,\cdots,h_r}$ with $r\geq 2$, it suffices to understand isomorphisms of short exact sequences of the form
\[
1\rightarrow N\rightarrow \Lambda_{g}\xrightarrow[]{\nu_g} \ZZ^k\rightarrow 1.
\]

The latter reduces to Linear Algebra. We will explain this fact briefly. A detailed exposition of the subject can be found in \cite{FarMar-12} (see in particular Chapter 6). In the following let $S_g$ be a closed surface of genus $g$ and let $\Lambda_g=\pi_1(S_g)$ be its fundamental group.

Let $MCG^{\pm}(S_g)$ be the \textit{(extended) mapping class group} of $S_g$, that is, the group of homeomorphisms of $S_g$ up to homotopy equivalences, where by the extended mapping class group we mean that we allow orientation reversing homeomorphisms. Let $Inn(\Lambda_g)$ be the group of \textit{inner automorphisms} of $\Lambda_g$, that is, automorphisms of the form $g\mapsto h^{-1}gh$ for a fixed $h\in \Lambda_g$ and let $Out(\Lambda _g)=Aut(\Lambda_g)/Inn(\Lambda_g)$ be the group of \textit{outer automorphisms} of $\Lambda _g$. 

The map $\Lambda_g\rightarrow \ZZ^k$ splits through the abelianisation $H_1(\Lambda_g,\ZZ)$ of $\Lambda_g$. Every automorphism $\tau$ of $\Lambda_g$ induces an automorphism $\tau_*\in GL(2g,\ZZ)$ of the abelianisation $H_1(\Lambda_g,\ZZ)$. Since inner automorphisms act trivially on $H_1(\Lambda_g,\ZZ)$, the induced homomorphism $Aut(\Lambda_g)\rightarrow GL(2g,\ZZ)$ splits through $Out(\Lambda_g)$. 

By the Dehn-Nielsen-Baer Theorem (cf. \cite[Theorem 8.1]{FarMar-12}) the natural map $MCG^{\pm}(S_g) \rightarrow Out(\Lambda _g)$ is an isomorphism. In particular, we can realise any element of $Out(\Lambda_g)$ by a (up to homotopy) unique homeomorphism of $S_g$. It follows that for $\tau \in Aut(\Lambda _g)$, the map $\tau_*\in GL(2g,\ZZ)$ is realised by some homeomorphism $\alpha\in MCG^{\pm}(S_g)$. 

There is a natural symplectic form on $H_1(\Lambda_g,\ZZ)$ induced by taking intersection numbers of representatives in $S_g$. Orientation preserving homeomorphism $\alpha \in MCG^+(S_g)$ of $S_g$ preserve intersection numbers. Hence, the induced automorphism $\alpha_*\in H_1(\Lambda_g,\ZZ)$ preserves the symplectic form which is equivalent to $A:=\alpha_*\in Sp(2g,\ZZ)$, where $Sp(2g,\ZZ)$ is the group of \textit{symplectic matrices} of dimension $2g$ with integer coefficients. It is defined by

\[
Sp(2g,\ZZ)=\left\{A\in M_{2g}(\ZZ) \mid A^tJ_{2g}A=J_{2g}\right\},
\]

for $J_{2g}$ the matrix representing the standard symplectic form given by the block diagonal matrix

\[
J_{2g}=\left(\begin{array}{ccccc}
J&0&\cdots& &0\\
0&J&0& \cdots &0\\
0&0&\ddots&\cdots &0\\
\vdots&\vdots &&\cdots &0\\
0& \cdots &\cdots&0&J\\
\end{array}\right)
\]
with $J_2=J=\left(\begin{array}{cc} 0&1\\ -1 & 0\end{array}\right)$ the standard symplectic form on $\RR ^2$.

For an orientation reversing homeomorphism we have that $A^tJA = -J$. We define the generalised symplectic group of dimension $2g$ with integer coefficients by

\[
Sp^{\pm}(2g,\ZZ)= \left\{A\in M_{2g}(\ZZ) \mid A^tJA=J \mbox{  or  } A^tJA=-J \right\}.
\]

As a result, there is a natural homomorphism $$\Psi: MCG^{\pm}(S_g)\rightarrow Sp^{\pm}(2g,\ZZ).$$

\begin{theorem}[{\cite[Theorem 6.4]{FarMar-12}}]
The symplectic representation $\Psi: MCG^{\pm}(S_g) \rightarrow Sp^{\pm}(2g,\ZZ)$ is surjective for $g \geq 1$.
\label{thmMCGSp}
\end{theorem}

Combining this with the isomorphism $MCG^{\pm}(S_g)\rightarrow Out(S_g)$, we obtain

\begin{corollary}
For $g\geq 1$, the symplectic representation $\Psi$ induces a surjective representation $Out(S_g)\rightarrow Sp^{\pm}(2g,\ZZ)$. 

In particular, two short exact sequences
\[
1\rightarrow \mathrm{ker}(\phi)\rightarrow \Lambda_{g}\xrightarrow[]{\phi} \ZZ^{k}\rightarrow 1
\]
and
\[
1\rightarrow \mathrm{ker}(\psi)\rightarrow \Lambda_{g}\xrightarrow[]{\psi} \ZZ^{k}\rightarrow 1
\]
are isomorphic if and only if there exists $A\in Sp^{\pm}(2g,\ZZ)$ and $A'\in GL(k,\ZZ)$ such that the following diagram commutes
\[
\xymatrix{H_1(\Lambda_g,\ZZ) \ar[r]^{\phi_{ab}} \ar[d]^{A}  & \ZZ^k \ar[d] ^{A'}\\ 
	   H_1(\Lambda_g,\ZZ)   \ar[r]^{\psi_{ab}} & \ZZ^k \\}
\]
where $\phi_{ab}$ and $\psi_{ab}$ are the unique maps factoring $\phi$ and $\psi$ through their abelianisation.
\label{corIsomSESeq}
\end{corollary}

\begin{proof}
The first part is a direct consequence of Theorem \ref{thmMCGSp} and the isomorphism $MCG^{\pm}(S_g)\cong Out(\Lambda_g)$.

The only if in the second part follows directly from the fact that for any outer automorphism of $\Lambda _g$, the induced map on homology is in the generalised symplectic group. The if follows, since we can lift any $A\in Sp(2g,\ZZ)$ to an element $\alpha \in Aut(\Lambda_g)$ inducing $A$ by the first part.
\end{proof}

\section{Classification for purely branched maps}
\label{secClassfb}

It follows from Section \ref{secComparison}, that showing that our groups are not isomorphic to Dimca, Papadima and Suciu's groups amounts to comparing maps on abelianisations and therefore reduces to a question in Linear Algebra. 

We will use the following classification of purely branched maps:

\begin{proposition}
 Let $S_g$ be a closed surface of genus $g$, let $E$ be an elliptic curve and let $f_g: S_g\rightarrow E$ be a purely branched covering map. Then the following are equivalent:
 
 \begin{enumerate}
 \item There are standard symplectic generating sets $\alpha_1,\beta_1,\cdots,\alpha_g,\beta_g$ of $\pi_1 S_g$ and $\alpha,\beta$ of $\pi_1E$ with respect to which the induced map on fundamental groups is
 \[
 \begin{array}{cccc}
  f_{g\ast}: &\pi_1 S_g &\rightarrow &\pi_1E\\
  & \alpha_1,\cdots,\alpha_k & \mapsto & \alpha\\
  & \beta_1,\cdots,\beta_k &\mapsto & \beta\\
  & \alpha_{k+1},\cdots,\beta_{g} & \mapsto & 0\\
  & \beta_{k+1},\cdots,\beta_{g}& \mapsto & 0.\\
 \end{array}
 \]
 \item The map $f_g$ is a $k$-fold branched covering map.
 \end{enumerate}
 \label{thmClassification}
\end{proposition}

Proposition \ref{thmClassification} follows from the following result:
 \begin{lemma}
  Let $S_g$ be a closed Riemann surface of genus $g$ and let $E$ be an elliptic curve. Let $f: S_g \rightarrow E$ be a holomorphic $k$-fold purely branched covering map. Then there exist standard generating sets $\alpha_1,\beta_1,\cdots, \alpha_g,\beta_g$ of  $\pi_1 S_g$ and $\alpha, \beta$ of $\pi_1 E$ such that the induced map on fundamental groups is of the form described in Proposition \ref{thmClassification}(1).
  \label{lemClassification}
 \end{lemma}
 
 \begin{proof}
  Let $B\subset E$ be the finite branching set of $f$. Since $f$ is purely branched there are generators $\alpha,\beta : \left[0,1\right] \rightarrow E\setminus B$ such that every lift of $\alpha$ and $\beta$ with respect to the unramified covering $f:S_g \setminus f^{-1}(B)\rightarrow E\setminus B$ is a loop.
  
  We may further assume that the only intersection point of $\alpha$ and $\beta$ is the point $\alpha(0)=\beta(0)$ and that the intersection number $\iota(\alpha,\beta)=1$ with respect to the orientation induced by the complex structure on $E$. 
  
  Since $f|_{S_g\setminus f^{-1}(B)}$ is a $k$-fold unramified covering map, there are precisely $k$ lifts $\alpha_1,\cdots,\alpha_k$ of $\alpha$ and $\beta_1,\cdots,\beta_k$ of $\beta$ and we may choose them so that $\iota(\alpha_i,\beta_j)=\delta _{ij}$, $\iota(\alpha_i,\alpha_j)=0$ and $\iota(\beta_i,\beta_j)=0$.
  
  Then the images of $\alpha_1,\beta _1, \cdots, \alpha_k,\beta_k$ in $H_1(S_g,\ZZ)$ forms part of a standard symplectic basis with respect to the symplectic intersection form on $H_1(S_g,\ZZ)$. Extend by $\alpha_{k+1},\beta_{k+1},\cdots, \alpha_g,\beta_g$ to a standard symplectic basis of $H_1(S_g,\ZZ)$. We claim that with respect to this basis $f_{\ast}$ takes the desired form.
  
 We may assume that $\alpha_{k+1},\beta_{k+1},\cdots, \alpha_g,\beta_g$ are loops in $S_g\setminus f^{-1}(B)$. Since $\alpha_j$, $k+1\leq j \leq g$, forms part of a symplectic basis we have that its intersection number with any of the $\alpha_i,\beta_i$ with $1\leq i \leq k$ is zero. Since all of the lifts of $\alpha,\beta$ are given by $\alpha_1,\beta_1, \cdots,\alpha_k,\beta_k$ and the map $f$ is holomorphic, thus orientation preserving, it follows that the intersection numbers $\iota(f\circ \alpha_j, \alpha)$ and $\iota(f\circ \alpha_j, \beta)$ satisfy
 \[
  \iota (f\circ \alpha_j,\alpha)=\sum_{i=1}^k \iota(\alpha_j,\alpha_i)=0,
 \]
 \[
 \iota (f\circ \alpha_j,\beta)=\sum_{i=1}^k \iota(\alpha_j,\beta_i)=0.
 \]
 
 Nondegeneracy of the symplectic intersection form on $H_1(E,\ZZ)=\pi_1E=\left\langle \alpha,\beta\mid \left[\alpha,\beta\right] \right\rangle$ then implies that $f\circ \alpha_j=0$ in $\pi_1 E$. Similarly $f\circ \beta_j=0$ in $\pi_1 E$ for $k+1\leq j\leq g$. 
 
 Since by definition $f\circ \alpha_i=\alpha$ and $f\circ \beta_i=\beta$ for $1\leq i \leq k$, it follows that with respect to the standard generating sets $\alpha_1,\beta_1,\cdots, \alpha_g,\beta_g$ of $\pi_1S_g$ and $\alpha,\beta$ of $\pi_1 E$, the induced map on fundamental groups is indeed
 \[
 \begin{array}{cccc}
  f_{\ast}: &\pi_1 S_g &\rightarrow &\pi_1E\\
  & \alpha_1,\cdots,\alpha_k & \mapsto & \alpha\\
  & \beta_1,\cdots,\beta_k &\mapsto & \beta\\
  & \alpha_{k+1},\cdots,\beta_{g} & \mapsto & 0\\
  & \beta_{k+1},\cdots,\beta_{g}& \mapsto & 0.\\
 \end{array}
 \]
 \end{proof}

  \begin{figure}[ht]
\includegraphics[width=10cm,keepaspectratio]{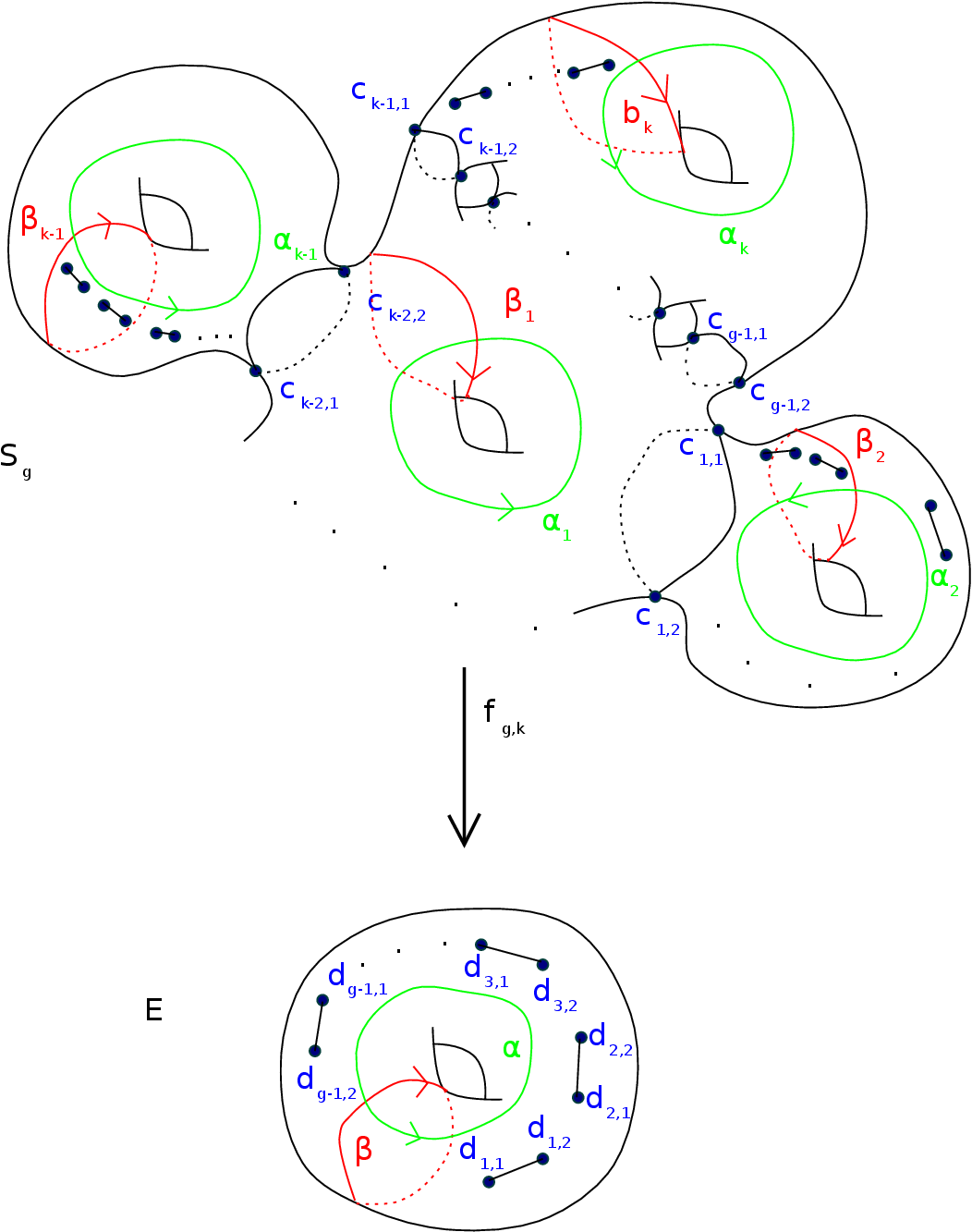}
\caption{The k-fold purely branched covering map $f_{g,k}$ of $E$ in Lemma \ref{lemkfoldfb}}
\label{figkfoldfb}
\end{figure} 

\begin{proof}[Proof of Proposition \ref{thmClassification}]
 By Lemma \ref{lemClassification} (2) implies (1). Assume that $f_{g,\ast}$ is of the form described in (1). 
 
The covering degree of $f_g$ can be obtained as
 \[
  \mathrm{deg}(f_g) = \left(f_g^{\ast}\alpha^{\ast} \cup f_g^{\ast} \beta^{\ast} \right) \cap \left[ S_g\right] =: \left\langle f_g^{\ast} \alpha^{\ast},f_g^{\ast} \beta^{\ast}\right\rangle,
 \]
 where $\cup$ denotes the cup product on cohomology, $\cap$ denotes the cap-product and $\left[S_g\right]\in H_2(S_g,\ZZ)$ is a fundamental class of $S_g$.
 
 Since $f_{g,\ast}$ is of the form described in (1) and $\alpha_1,\beta_1, \dots, \alpha_g,\beta_g$ is a standard symplectic generating set of $\pi_1 S_g$ and thus of $H_1(S_g,\ZZ)$, we obtain that 
 \[
 \mathrm{deg}(f_g)= \left\langle f_g^{\ast} \alpha^{\ast} , f_g^{\ast} \beta^{\ast} \right\rangle = \left\langle \sum_{i=1}^k \alpha_i^{\ast},\sum_{j=1}^k \beta_j^{\ast}\right\rangle = \sum _{i,j=1}^k \delta_{ij}=k.
 \]
 Thus $f_g$ is a $k$-fold branched covering map.
\end{proof}

\begin{remark}
 The elementary argument that (1) implies (2) in the proof Proposition \ref{thmClassification} was provided to us by the anonymous referee. In a previous version of this paper we proved this implication using the following argument: Proposition \ref{propNotIsom} and Corollary \ref{corIsomSESeq} imply that $k$ is an invariant of the purely branched covering map $f_g$ and, by Lemma \ref{lemClassification}, $k$ must then coincide with the degree of the map $f_g$.
\end{remark}

We show the existence of branched covering maps satisfying the equivalent conditions in Proposition \ref{thmClassification}.
 \begin{proposition}
  For every $g\geq 2$ and every $2\leq k \leq g$ there is a $k$-fold purely branched holomorphic covering map $f_{g,k}: S_g\rightarrow E$ with Morse type singularities.
  \label{lemkfoldfb}
 \end{proposition}
 \begin{proof}
 Let $E$ be an elliptic curve, let $g\geq 2$, let $1\leq k\leq g$ and let $d_{1,1},d_{1,2},d_{2,1},\cdots , d_{g-1,1},d_{g-1,2}$ be $2(g-1)$ points in $E$ and $s_1,\cdots, s_{g-1}: \left[0,1\right]\rightarrow  E$ be simple, pairwise non-intersecting paths with starting point $s_i(0)=d_{i,1}$ and endpoint $s_i(1)=d_{i,2}$ for $i=1,\cdots,g-1$. Take $k$ copies $E_0$, $E_1,\cdots, E_{k-1}$ of $E$, cut $E_0$ open along all of the paths $s_i$, cut $E_i$ open along the path $s_i$ for $1\leq i \leq k-2$ and cut $E_{k-1}$ open along the paths $s_{k-1},\cdots, s_{g-1}$. This produces surfaces $F_0,\cdots, F_{k-1}$ with boundary.
 
 Gluing the surfaces $F_0,\cdots, F_{k-1}$ in the unique way given by identifying opposite edges in the corresponding boundary components $F_0$ and each of the $F_i$ we obtain a closed surface of genus $g$ together with a continuous $k$-fold purely branched covering map $f_{g,k}: S_g \rightarrow E$. Choosing the unique complex structure on $S_g$ that makes $f_{g,k}$ holomorphic we do indeed obtain a $k$-fold purely branched holomorphic covering map $f_{g,k}: S_g \rightarrow E$ pictured in Figure \ref{figkfoldfb}. Looking at this map in local coordinates it is immediate that all singularities are of Morse type.
\end{proof}

We will obtain a proof of Theorem \ref{thmClassfbKG} by combining Proposition \ref{thmClassification} with the following result in Linear Algebra.

\begin{proposition}
For $g\geq 2$ and $1\leq k<l\leq g$ there are no linear maps $A\in Sp^{\pm}(2g,\RR)$ and $B\in Gl(2,\ZZ)=Sp^{\pm}(2,\ZZ)$ 
such that 
\begin{equation}
( \underbrace{I \cdots I}_\text{k times} \underbrace{0 \cdots 0}_\text{g-k times} )\cdot A
=
B\cdot( 
\underbrace{I \cdots I}_\text{l times} \underbrace{0 \cdots 0}_\text{g-l times}),
\label{eqnLinCondGen}
\end{equation}
where $I=I_2=\left(\begin{array}{cc} 1&0\\ 0&1 \end{array}\right)$ is the 2-dimensional identity matrix.
\label{propNotIsom}
\end{proposition}

\begin{proof}
The proof is by contradiction. 

Assume that there is $A \in Sp^{\pm}(2g,\RR)$ and $B\in Gl(2,\ZZ)=Sp^{\pm}(2,\ZZ)$ satisfying Equation \eqref{eqnLinCondGen}. Define $\gamma_1,\cdots,\gamma_{2g}\in \RR ^{2k}$ and $\alpha _1,\cdots, \alpha_{2g}\in \RR^{2g-2k}$ by
 \[
 A=\left(\begin{array}{cccc}
 \gamma_1 & \gamma_2 & \cdots & \gamma_{2g}\\
 \alpha_1 & \alpha_2 & \cdots & \alpha_{2g}\\
 \end{array}\right).
 \]
 
 Then $A\in Sp^{\pm}(2g,\RR)$ implies that
{\small
 \begin{align*}
 \pm J_{2g}=A^tJA = &
 \left(\begin{array}{cc}
 \gamma_1^t &\alpha_1 ^t\\
 \vdots&\vdots\\
 \gamma_{2g}^t &\alpha_{2g}^t\\
 \end{array}\right)\cdot
 \left(\begin{array}{cc} J_{2k}&0\\ 0&J_{2g-2k}\\\end{array}\right)\cdot
 \left(\begin{array}{cccc}
 \gamma_1 & \gamma_2 & \cdots & \gamma_{2g}\\
 \alpha_1 & \alpha_2 & \cdots & \alpha_{2g}\\
 \end{array}\right)\\
 =& \left(\begin{array}{c}\gamma_1^t\\ \vdots \\ \gamma_{2g}^t\end{array}\right)\cdot J_{2k}\cdot \left(\gamma_1\cdots \gamma_{2g}\right) + \left(\begin{array}{c}\alpha_1^t\\ \vdots \\ \alpha_{2g}^t\end{array}\right) \cdot J_{2g-2k} \cdot \left(\alpha_1\cdots \alpha_{2g}\right)\\
 =& \underbrace{\left[\gamma_i^t\cdot J_{2k}\cdot \gamma_j\right]_{i,j=1,\cdots,2g}}_{=:E} + \underbrace{\left[\alpha_i^t\cdot J_{2g-2k}\cdot \alpha_j\right]_{i,j=1,\cdots, 2g}}_{=:F}\\
 \end{align*}}
 The map $E$ is of rank $\leq 2k$, since it splits through $\RR^{2k}$ and the map $F$ is of rank $\leq 2g-2k$, since it splits through $\RR^{2g-2k}$.
 
 We will now prove that in fact $E$ is of rank $\leq 2k-2$. Equation \eqref{eqnLinCondGen} implies that for 
 \[
  \left(\gamma_1 \cdots \gamma_{2g}\right)=\left(\begin{array}{ccc} A_{11}&\cdots& A_{1g}\\ \vdots & \ddots & \vdots\\ A_{k1} & \cdots & A_{kg} \\ \end{array} \right),
 \]
 where $A_{ij}\in \RR^{2\times 2}$ for $1\leq i\leq k$ and $1\leq j\leq g$, we have
 \[
  (\underbrace{I\cdots I}_\text{k times})\cdot \left(\begin{array}{ccc} A_{11}&\cdots& A_{1g}\\ \vdots & \ddots & \vdots\\ A_{k1} & \cdots & A_{kg} \\\end{array} \right) = (\underbrace{B\cdots B}_\text{l times} \underbrace{0\cdots 0}_\text{g-l times}).
 \]
 
 It follows that
 \[
 \sum_{i=1}^k A_{ij}=\left\{ \begin{array}{ll} B &\text{if $j\leq l$}\\ 0 & \text{if $j>l$}\\\end{array} \right.
 \]
 and hence, that

 \[
 \begin{array}{ll}
  \left(\gamma_1 \cdots \gamma_{2g}\right)&= \mbox{ \resizebox{8cm}{.75cm}{ $
  \left(\begin{array}{cccccc} A_{11}&\cdots& A_{1l} & A_{1(l+1)} &\cdots & A_{1g}\\ 
  \vdots & \ddots & \vdots &\vdots & \ddots & \vdots\\ 
  A_{(k-1)1} & \cdots& A_{(k-1)l} &A_{(k-1)(l+1)} & \cdots & A_{(k-1)g} \\
  B-\sum_{i=1}^{k-1} A_{i1} &\cdots & B-\sum_{i=1}^{k-1} A_{il} & -\sum_{i=1}^{k-1} A_{i(l+1)}& \cdots
 & -\sum_{i=1}^{k-1} A_{ig} \end{array} \right)$ }}\\ &=: \left(M N\right)
 \end{array}
 \]

 with $M\in \RR^{2k\times 2l}$ and $N\in \RR^{2k\times 2(g-l)}$.
 
 Then, we have
 \[
 E= \left(\begin{array}{cc} M^t J_{2k} M & M^t J_{2k} N\\ N^t J_{2k} M & N^t J_{2k} N\\\end{array}\right).
 \]
 
 Define linear maps
 {\small \[
  \begin{array}{ll}
   M'=&\left[\mathrm{det}B\cdot \left( B^{-1}A_{ij}-\frac{1}{\sqrt{k}-1} \sum_{l=1}^{k-1}B^{-1}A_{lj} + \frac{I_2}{\sqrt{k}}\right)\right]_{i=1,\cdots,k-1, j=1,\cdots,l}\\ & \in \RR^{2(k-1)\times 2l}
  \end{array}
 \]
 \[
 \begin{array}{ll}
  N' =& \left[\mathrm{det B}\cdot \left(B^{-1}A_{im} - \frac{1}{\sqrt{k}-1} \sum_{r=1}^{k-1} B^{-1}A_{rm}\right)\right]_{i=1,\cdots, k-1,m=l+1,\cdots, g}\\ &\in \RR^{2(k-1)\times 2(g-l)}.
 \end{array}
 \]}
 
 Using that $B^tJ_2B=\mathrm{det}B \cdot J_2$, we obtain the following identities
 \[
 M^t J_{2k} M= \frac{\mathrm{det} B}{k} \cdot \left(\begin{array}{ccc} J& \cdots& J\\ \vdots & \ddots & \vdots \\ J & \cdots & J \end{array}\right) + M'^t J_{2k-2} M'
 \]
 \[
 M^t J_{2k} N = M'^t J_{2k-2} N'
 \]
 \[
  N^t J_{2k} M = N'^t J_{2k-2} M'
 \]
 \[
  N^t J_{2k} N = N'^t J_{2k-2} N'
 \]
 
 In particular, this implies that
 \[
 \begin{array}{cl}
  E = \left[\gamma_i^t J_{2k}\gamma_j\right]_{i,j=1,\cdots,2g} = &
  \left(\begin{array}{cc}   \begin{array}{ccc} \frac{\mathrm{det} B}{k}J& \cdots& \frac{\mathrm{det} B}{k}J\\ \vdots & \ddots & \vdots \\ \frac{\mathrm{det} B}{k}J & \cdots & \frac{\mathrm{det} B}{k}J \end{array} &0\\ 0 & 0 \\\end{array}\right)\\ &+\underbrace{ \left(\begin{array}{cc} M'^t \\ N'^t\end{array}\right) \cdot J_{2k-2} \cdot \left(M' N'\right)}_{=:S}.
  \end{array}
 \]
 
 The linear map $S$ splits through $\RR^{2k-2}$, implying that $\mathrm{rank}(S)\leq 2k-2$. Furthermore, $\pm J_{2g} = E + F$ and $\mathrm{det}B=\pm 1$ imply that

 \begin{equation}
 \begin{array}{ll}
  F &= \pm J_{2g} - E\\ &=  \mbox{ {\resizebox{8cm}{1.4cm}{ $\underbrace{\left(\begin{array}{cc} 
  \begin{array}{ccccc} 
  (\pm 1 \pm \frac{1}{k})J& \pm \frac{1}{k}J & \cdots& \cdots & \pm \frac{1}{k}J\\
     \pm \frac{1}{k}J &\ddots & \ddots & & \vdots \\
 	\vdots & \ddots & \ddots & \ddots & \vdots \\
 	\vdots & & \ddots & \ddots & \pm \frac{1}{k}J\\   
    \pm \frac{1}{k}J & \cdots & \cdots & \pm \frac{1}{k}J& (\pm 1\pm \frac{1}{k})J\\ \end{array}
   &0\\ 0 & \pm J_{2(g-l)} \\\end{array}\right)}_{=:R} - S$} }}
  \end{array}
\label{eqnInvbleR} 
 \end{equation}

 Hence, $F$ has rank $\geq 2g-(2k-2)= 2(g-k) + 2 $, since by Lemma \ref{lemInvble} below $R$ is invertible for $l\neq k$. This contradicts $\mathrm{rank}F\leq 2(g-k)$, showing that there are no $A\in Sp^{\pm}(2g,\RR)$ and $B\in Gl(2,\ZZ)$ satisfying \eqref{eqnLinCondGen}. 
\end{proof}

\begin{lemma}
 For $l\in \ZZ$, $k\in \RR$, the linear map $R$ defined in \eqref{eqnInvbleR} is invertible if and only if $l\neq \pm k$. 
 \label{lemInvble}
\end{lemma}
 \begin{proof}
 Clearly, it suffices to prove that for $k\neq l$, the matrix
 \[
 R_{2l}=\left(
 \begin{array}{ccccc} 
  (\pm 1 \pm \frac{1}{k})J& \pm \frac{1}{k}J & \cdots& \cdots & \pm \frac{1}{k}J\\
     \pm \frac{1}{k}J &\ddots & \ddots & & \vdots \\
 	\vdots & \ddots & \ddots & \ddots & \vdots \\
 	\vdots & & \ddots & \ddots & \pm \frac{1}{k}J\\   
    \pm \frac{1}{k}J & \cdots & \cdots & \pm \frac{1}{k}J& (\pm 1\pm \frac{1}{k})J\\ \end{array}
 \right)\in \RR^{2l\times 2l}
 \]
 is invertible. We do row and column operations in order to compute the rank of $R_{2l}$. 
 
 Subtracting the last (double) row from all other rows yields
 \[
 \left(
 \begin{array}{cccccc} 
  \pm J&  0 & \cdots & 0 & \mp J\\
     0 & \pm J &\ddots &  \vdots & \vdots \\
 	\vdots & \ddots & \ddots & 0 & \vdots \\
 	0 &\cdots  & 0 & \pm J  & \mp  J\\   
    \pm \frac{1}{k}J & \cdots &\cdots & \pm \frac{1}{k}J& (\pm 1\pm \frac{1}{k})J\\ \end{array}
 \right)
 \] 
 
 After subtracting multiples of the first $(l-1)$ (double) rows from the last row we obtain
 \[
  \left(
 \begin{array}{cccccc} 
  \pm J&  0 & \cdots & 0 & \mp J\\
     0 & \pm J &\ddots &  \vdots & \vdots \\
 	\vdots & \ddots & \ddots & 0 & \vdots \\
 	0 &\cdots  & 0 & \pm J  & \mp  J\\   
    0 & \cdots &\cdots & 0& (\pm 1\pm (\frac{1}{k}+\frac{l-1}{k}))J\\ \end{array}
 \right).
 \]
 Hence, $R$ is invertible if and only if $\pm 1 \pm \frac{l}{k} \neq 0$. This is clearly the case for all choices of signs if and only if $l\neq \pm k$, completing the proof.
 \end{proof}
 
\begin{proof}[Proof of Theorem \ref{thmClassfbKG}]
 Theorem \ref{thmFinProp} implies that the groups $\pi_1 H_p$ and $\pi_1 H_q$ are kernels of short exact sequences of the form described in Proposition \ref{thmAbstractIsom}. Thus, they are isomorphic if and only if the corresponding short exact sequences are. Since by assumption all maps are purely branched covering maps, the induced maps on fundamental groups are of the form described in Proposition \ref{thmClassification}(1). Hence, the result follows from Proposition \ref{thmIsomSESeq}, Corollary \ref{corIsomSESeq} and Proposition \ref{propNotIsom}.
\end{proof}

It follows that our groups are indeed not isomorphic to Dimca, Papadima, and Suciu's groups apart from the obvious isomorphism for $r=s$ and $g_1=\cdots=g_r=h_1=\cdots=h_r=2$:

\begin{corollary}
 The K\"ahler groups $\mathrm{ker}(\phi_{g_1,\cdots,g_r})$ obtained from \resizebox{18pt}{7pt}{ \eqref{eqnDPSmap}} and $\mathrm{ker}(\psi_{h_1,\cdots,h_r})$ obtained from \eqref{eqnOurExamples} are isomorphic if and only if $r=s$ and $g_1=\cdots=g_r=h_1=\cdots=h_r=2$.
\end{corollary}
\begin{proof}
This is an immediate consequence of Theorem \ref{thmClassfbKG} and the fact that we constructed our groups as fundamental groups of the fibre of a sum of $h_i$-fold purely branched holomorphic maps in Section \ref{secSpecConst}, while, as we also saw in Section \ref{secSpecConst}, Dimca, Papadima and Suciu's groups arise as fundamental groups of the fibre of a sum of $2$-fold purely branched holomorphic maps.
\end{proof}

\begin{remark}
Note that Proposition \ref{thmAbstractIsom} and Proposition \ref{thmClassification} also allow us to distinguish K\"ahler groups arising from our construction for which not all maps are purely branched, provided that the purely branched maps do not coincide on fundamental groups up to reordering and choosing suitable standard generating sets.

It seems reasonable to us that there is a further generalisation of Proposition \ref{propNotIsom} to branched covering maps which are not purely branched. A suitable generalisation would allow us to classify all K\"ahler groups that can arise using our construction up to isomorphism. 
\end{remark}
 
We want to conclude by giving an example of a branched covering map which is surjective on fundamental groups, but is not purely branched:

\begin{example}
\label{ex:surjnotpb}
 Let $f_2: S_2 \rightarrow E$ be the 2-fold purely branched covering map in Dimca, Papadima and Suciu's construction, let $B \subset E$ be its branching locus and let $a,b : \left[0,1\right] \rightarrow E\setminus B$ be any choice of simple closed generators for $\pi_1 E$ with the properties that $a$ and $b$ intersect in a single point and all lifts of $a$ and $b$ are loops. Lemma \ref{lemClassification} and its proof imply that there is a generating set $a_1,b_1,a_2,b_2$ of $\pi_1 S_2$ such that the induced map $\pi_1 S_2 \rightarrow \pi_1 E$ on fundamental groups is given by $a_i\mapsto a$, $b_i\mapsto b$ and the set of all lifts of the loops $a$ and $b$ to $S_2$ forms a symplectic generating set for the homology $H_2(S_2,\ZZ)$. 
 
 Now fix a choice of generators $a$ and $b$ as above and consider the unramified 2-fold covering map $p: S_3 \rightarrow S_2$ induced by the epimorphism $H_1(S_2,\ZZ)\rightarrow \ZZ/2\ZZ$, $a_1\mapsto 1$, $a_2,b_1,b_2\mapsto 0$. Then no lift of $a_1$ to $S_3$ is a loop and all lifts of $a_2$, $b_1$, $b_2$ to $S_3$ are loops. It follows that the branched covering map $f_2\circ p$ is surjective on fundamental groups. By the previous paragraph, any different choice of simple closed loops $\mu_1,\mu_2: \left[0,1\right]\rightarrow E\setminus B$, satisfying the conditions in the definition of a purely branched covering map, lifts to a symplectic generating set $\mu_{1,1},\mu_{2,1},\mu_{1,2},\mu_{2,2}$ of $H_1(S_2,\ZZ)$. Since the induced map $p_{\ast} : H_1(S_3,\ZZ)\rightarrow H_1(S_2,\ZZ)$ on homology is not surjective, there is a lift of one of the $\mu_{i,j}$ which is not a loop in $S_3$. It follows that $f_2\circ p$ is not a purely branched covering.
\end{example}

\bibliography{References}

\providecommand{\bysame}{\leavevmode\hbox to3em{\hrulefill}\thinspace}
\providecommand{\MR}{\relax\ifhmode\unskip\space\fi MR }
\providecommand{\MRhref}[2]{%
  \href{http://www.ams.org/mathscinet-getitem?mr=#1}{#2}
}
\providecommand{\href}[2]{#2}
\begin{thebibliography}{10}

\bibitem{BesBra-97}
M.~Bestvina and N.~Brady, \emph{Morse theory and finiteness properties of
  groups}, Invent. Math. \textbf{129} (1997), no.~3, 445--470.

\bibitem{BisMjPan-14}
I.~Biswas, M.~Mj, and D.~Pancholi, \emph{{Homotopical Height}}, Int. J. Math.
  Vol. \textbf{25} (2014), no.~13.

\bibitem{BriHow-07}
M.R. Bridson and J.~Howie, \emph{Normalisers in limit groups}, Math. Ann.
  \textbf{337} (2007), no.~2, 385--394.

\bibitem{BriHowMilSho-02}
M.R. Bridson, J.~Howie, C.F. Miller~III, and H.~Short, \emph{{The subgroups of
  direct products of surface groups}}, Geometriae Dedicata \textbf{92} (2002),
  95--103.

\bibitem{BriHowMilSho-13}
M.R. Bridson, J.~Howie, C.F. Miller~III, and H.~Short, \emph{{On the finite presentation of subdirect products and the nature
  of residually free groups}}, American Journal of Math. \textbf{135} (2013),
  no.~4, 891--933.

\bibitem{Bro-82}
K.S. Brown, \emph{Cohomology of groups}, Graduate Texts in Mathematics,
  vol.~87, Springer-Verlag, New York-Berlin, 1982.

\bibitem{DimPapSuc-08}
A.~Dimca, S.~Papadima, and A.I. Suciu, \emph{Quasi-{K}\"ahler
  {B}estvina--{B}rady groups}, J. Algebraic Geom. \textbf{17} (2008), no.~1,
  185--197.

\bibitem{DimPapSuc-09-II}
A.~Dimca, S~Papadima, and A.I. Suciu, \emph{Non-finiteness properties of
  fundamental groups of smooth projective varieties}, J. Reine Angew. Math.
  \textbf{629} (2009), 89--105.

\bibitem{FarMar-12}
B.~Farb and D.~Margalit, \emph{{A primer on Mapping Class groups}}, Princeton
  Math. Series, vol.~49, Princeton University Press, Princeton, NJ, 2012.

\bibitem{Jac-70}
W.~Jaco, \emph{On certain subgroups of the fundamental group of a closed
  surface}, Proc. Cambridge Philos. Soc. \textbf{67} (1970), 17--18.

\bibitem{Llo-16}
C.~Llosa~Isenrich, \emph{Finite presentations for {K}\"ahler groups with
  arbitrary finiteness properties}, J. Algebra \textbf{476} (2017), 344--367.

\bibitem{Suc-12}
A.I. Suciu, \emph{Characteristic varieties and {B}etti numbers of free abelian
  covers}, Int. Math. Res. Not. IMRN (2014), no.~4, 1063--1124.

\end{thebibliography}
\bibliographystyle{amsplain}

\end{document}